\documentclass[a4paper,12pt]{article}
\usepackage[tbtags, sumlimits, intlimits]{amsmath}
\usepackage{amsfonts}
\usepackage{amssymb}
\usepackage[amsmath,thmmarks, thref]{ntheorem}
\theoremstyle{changebreak}
\usepackage{float}
\usepackage{verbatim}
\usepackage{array}
\usepackage{supertabular}
\usepackage{moreverb}
\usepackage{dsfont}

\newcommand{\R}{\ensuremath{\mathbb{R}}} 
\newcommand{\E}{\ensuremath{\mathbb{E}}}

\newcommand{\Z}{\ensuremath{\mathbb{Z}}}

\newcommand{\F}{\ensuremath{\mathbb{F}}}

\newcommand{\s}{\ensuremath{\mathcal{S}}}
\newcommand{\OO}{\ensuremath{\mathcal{O}}}

\newcommand{\disj}{\ensuremath{\stackrel{\cdot}{\bigcup}}}

\newtheorem{theorem}{Theorem}[section]
\newtheorem{defin}[theorem]{Definition}
\newtheorem{rem}[theorem]{Remark}
\newtheorem{lemma}[theorem]{Lemma}
\newtheorem{cor}[theorem]{Corollary}
\theoremstyle{nonumberplain}
\theoremheaderfont{\itshape}
\theorembodyfont{\normalfont}
\theoremseparator{:}
\theoremprework{\relax}
\theorempostwork{\relax}
\qedsymbol{\ensuremath{_\square}}
\theoremsymbol{\ensuremath{_\square}}
\newtheorem{proof}{Proof}

\begin{document}
\title{Spherical designs and lattices}
\author{Elisabeth Nossek}
\maketitle
\begin{abstract}
In this article we prove that integral lattices with minimum $\le 7$
(or $\le 9$) whose set of minimal vectors form spherical $9$-designs
(or $11$-designs respectively) are extremal, even and unimodular.
We furthermore show that there does not exist an integral lattice
with minimum $\le 11$ which yields a $13$-design.
\end{abstract}
\section{Introduction}

The density of a sphere packing associated to a lattice $\Lambda$
is given through the Hermite function $\gamma(\Lambda)$.
The local maxima of $\gamma$ are called extreme lattices and 
were characterised through the geometry
of their shortest vectors,
$S(\Lambda):=\{l\in \Lambda|(l,l)=\min(\Lambda)\}$, where 
$\min(\Lambda):=\min\{(x,x)|0\not=x\in \Lambda\}$, in
the works of Voronoi(\cite{Voronoi}), Korkine and Zolotareff(\cite{KoZ}).
A prominent subclass of extreme lattices are the strongly perfect
lattices introduced by Venkov \cite{venkov}. They are
characterised by the property that $S(\Lambda)$ forms a spherical
$5$-design:

\begin{defin}
A finite subset $X$ of the $n$-dimensional sphere $\s^{n-1}(m)$ of radius $m$ 
forms a spherical t-design if 
\[\int_{\s^{n-1}(m)}f(x) dx = \frac{1}{|X|}\sum_{x\in X}f(x)\]
for all homogeneous polynomials $f$ in $n$ Variables and of degree $\le t$.
A lattice $\Lambda$ such that $S(\Lambda)$ is a spherical $t$-design
is called a $t$-design lattice.
\end{defin}

The classification of strongly perfect lattices is 
known up to dimension $12$ (\cite{dim10}, \cite{dim12}),
but becomes very complicated in higher dimensions (see \cite{mythesis}).
Venkov \cite{venkov} and Martinet \cite{martinet} imposed further 
design conditions and classified all integral lattices of $\min\le 3$
(resp. $\min\le 5$) whose minimal vectors form spherical $5$-designs
(resp. $7$-designs).

This paper extends their work, more precisely we prove the following 
theorem:
\begin{theorem}\label{main}
\begin{enumerate}
\item The only integral $9$-design lattices with minimum $\le 7$ are
	the Leech lattice $\Lambda_{24}$ and 
	the extremal even unimodular lattices in dimension $48$.\label{a}
\item The only integral $11$-design lattices with minimum $\le 9$ are 
	$\Lambda_{24}$ and the $48$ and $72$ dimensional extremal even
	unimodular lattices.\label{b}
\item There is no integral $13$-design lattice with minimum $\le 11$.\label{c}
\end{enumerate}
\end{theorem}

\section{Some facts about spherical designs and lattices}
As $9$ and $11$-designs are also $7$-designs,
we will summarize their classification known from \cite{martinet}:
\begin{theorem}\label{min_le_5}
The integral $7$-design lattices with minimum $\le 5$ are $\E_8$, 
the unimodular lattice $\OO_{23}$ with minimum $3$, 
the three laminated lattices $\Lambda_{16}$ (the Barnes-Wall
lattice), $\Lambda_{23}$ and $\Lambda_{24}$ (the Leech lattice)
and the unimodular lattices of dimension $32$ and minimum $4$.
\end{theorem}
Martinet also proves that only the Leech lattice 
is an $11$-design lattice and the other lattices in Theorem \ref{min_le_5} do not 
yield $8$-designs \cite[Proposition]{martinet}.
Hence the only integral lattice with minimum $\le 5$ whose
minimal vectors form a $9$ or $11$-design is the Leech lattice.

In this article we will use the following characterisation
(see \cite[th. 3.2]{venkov}):
\begin{theorem}\label{char}
A finite set $X=-X\subset \s^{n-1}(m)$ forms a spherical 
$2t+1$-design if and only if 
\begin{align*}
D_{2i}(\alpha)&:=\sum_{x\in X} (x,\alpha)^{2i}=c_{i}|X|m^i(\alpha,\alpha)^i\\
\text{with } c_i&:=\prod_{k=0}^{i-1}\frac{1+2k}{n-2k}
\end{align*}
holds for all $i\le t$ and all $\alpha\in\R^n$.
\end{theorem}


In the following we will often distinguish between unimodular 
and non-unimodular lattices. If $\Lambda$ is an integral non-unimodular
lattice then
for $v\in\Lambda^*$ minimal in its class modulo
$\Lambda$ holds that $|(v,\lambda)|\le \frac{\min(\Lambda)}{2}$
for all $\lambda \in S(\Lambda)$ (\cite[Lemme 1.1]{martinet}).
For even  non-unimodular lattices $\Lambda$ 
we know that $\Lambda^*/\Lambda$ is a regular quadratic group in 
particular there exists an element $w\in \Lambda^*$ with 
$(w,w)\not\in 2\Z$ and we can assume w.l.o.g. that such a $w$ is minimal in
its class. 


\section{$9$-design lattices of minimum $\le 7$}
Throughout this section $\Lambda\subseteq \R^n$ denotes an 
integral $9$-design lattice of minimum $m\le 7$ with $X\disj -X:=S(\Lambda)$
and $s:=|X|$.\\
We will start by proving part \ref{a} of Theorem \ref{main}.
The characterisation in Theorem \ref{char} leads to the 
following system of linear 
equations for which only integral solutions correspond to integral $9$-design 
lattices.

\begin{lemma}\label{9-des}
For all $\alpha\in S(\Lambda)$ put 
$s_i(\alpha):=|\{x\in X|(x,\alpha)=\pm i\}|$.
The $s_i$ are independent of $\alpha$ and $s_i=0$ for $i>3$.
The following system of linear equations has non-negative integral solutions for 
the $s_i$ and for $s$ if $S(\Lambda)$ is a spherical $9$-design:
\begin{align*}
\begin{pmatrix} 1&2^2&3^2\\1&2^4&3^4\\1&2^6&3^6\\1&2^8&3^8
\end{pmatrix}
\begin{pmatrix}s_1\\ s_2\\ s_3\\ \end{pmatrix}=
\begin{pmatrix}\frac{sm^2}{n} -m^2\\ \frac{3sm^4}{n(n+2)}-m^4\\
	\frac{15sm^6}{n(n+2)(n+4)}-m^6\\ \frac{105sm^8}{n(n+2)(n+4)(n+6)}-m^8
\end{pmatrix}.
\end{align*}
\end{lemma}
\begin{proof}
The system of equations is just a result of the evaluation of 
the equations in Theorem \ref{char} for $\alpha\in S(\Lambda)$.
\end{proof}

\begin{rem}
A simple calculation with Pari shows that 
for $m=7$ there are no non-negative integral solutions 
$(n,s,s_1,s_2,s_3)\in\Z_{>0}^5$.
For $m=6$ non-negative integral solutions exist only for the following values of
$n$ and $s$:
\begin{table}[h]
\begin{center}
\begin{tabular}{|c|c|c|c|c|c|c|}
\hline
n & 26    & 36      & 44      & 46       & 48       & 49\\
\hline
s & 69888 & 1149120 & 8500800 & 13395200 & 26208000 & 50992095\\
\hline
\end{tabular}
\end{center}
\caption{Dimensions and kissing numbers for integral 9-design lattices.}
\label{9designs}
\end{table}
\end{rem}

Following a method used in \cite{martinet} we will have a look at non-unimodular
lattices at first.
\begin{lemma}\label{26_36}
If $\Lambda$ is non-unimodular and $\min(\Lambda)=6$ then $n\in\{26,36\}$. 
\end{lemma}
\begin{proof}
For all elements $v \in \Lambda^*\setminus\Lambda$ that are minimal 
in their class modulo $\Lambda$ we can define
$t_i(v):=|\{x\in S(\Lambda)|(x,v)=i\}|$. The $t_i$ are independent 
of $v$ and for $i>4$ $t_i=0$.
Therefore we get a system of equations again with $t:=(v,v)$:
\begin{align*}
\begin{pmatrix} 1&2^2&3^2\\1&2^4&3^4\\1&2^6&3^6\\1&2^8&3^8
\end{pmatrix}
\begin{pmatrix} t_1\\ t_2\\ t_3\\ \end{pmatrix}=
\begin{pmatrix} \frac{smt}{n}\\ \frac{3sm^2t^2}{n(n+2)}\\
	\frac{15sm^3t^3}{n(n+2)(n+4)}\\ \frac{105sm^4t^4}{n(n+2)(n+4)(n+6)}
\end{pmatrix}.
\end{align*}
$t$ has to be rational and positive. For every pair $(n,s)$ from 
Table \ref{9designs} we get a solution of the system and a polynomial equation 
$p_n$ of degree $4$ whose positive rational roots are the possible values for $t$.
But the only cases in which $p_n$ has such roots are $n=26$
where $t\in\{\frac{8}{3},4\}$ and for $n=36$ where $t=4$.
\end{proof}

\begin{lemma}\label{not26_36}
There is no non-unimodular lattice in dimension $26$ or $36$ such that
its set of minimal vectors form a spherical $9$-design.
\end{lemma}
\begin{proof}
Let $\Lambda$ be a non-unimodular lattice.
Without loss of generality we can assume that $\Lambda$ is generated by its
minimal vectors, hence $\Lambda$ is even.
For $n=36$ we know that $(v,v)=4$ for all $v$ in 
$\Lambda^*\setminus \Lambda$ with minimal norm in its class modulo
$\Lambda$. Hence $\Lambda^*$ has to be even and therefore unimodular which 
contradicts our assumption.

For $n=26$ we know that $(v,v)\in\{\frac{8}{3},4\}$ 
for $v$ in $\Lambda^*\setminus \Lambda$ 
with minimal norm in its class modulo $\Lambda$.
$\Lambda^*/\Lambda$ is a regular quadratic $\F_3$ space with 
$q:\Lambda^*/\Lambda\rightarrow \F_3$ with 
$q(x+\Lambda):=\frac{3(x,x)}{2}\mod 3$. 
Because $q(\Lambda^*/\Lambda)=\{0,1\}$ we know that $\Lambda^*/\Lambda$
is an one-dimensional $\F_3$ space with a generator $v$ with $q(v)=1$. 
Hence $\det(\Lambda)=3$ and $\gamma(\Lambda)=\frac{6}{3^{1/26}}$ which 
is greater than the Hermite constant $\gamma_{26}$ (see \cite[Table 3]{cohelk}).
\end{proof}

\begin{lemma}\label{unimod}
If $\Lambda$ is unimodular and $\min(\Lambda)=6$ then $n=48$ and $\Lambda$
is even and extremal.
\end{lemma}
\begin{proof}
Let $\Lambda^{(e)}:=\{\lambda\in \Lambda|(\lambda,\lambda)\in 2\Z\}$ be the even
sublattice of $\Lambda$ then $S(\Lambda^{(e)})=S(\Lambda)$ and 
$\Lambda^{(e)}$ is even and unimodular as a result of Lemma \ref{26_36}
and Lemma \ref{not26_36}. Therefore $n$ has to be divisible by
$8$ as a result of a theorem by Hecke (see e.g. \cite[Satz V.2.5]{krieg}),
hence $n=48$. As $\Lambda^{(e)}$ is unimodular it has to be equal to 
$\Lambda$, so $\Lambda$ is even and obviously extremal.
\end{proof}
This concludes the proof of Theorem \ref{main} Part \ref{a}.

\begin{cor}
Both the Leech lattice and the $48$-dimensional even unimodular lattices 
yield not only $9$-designs but also $11$-designs.
\end{cor}
\begin{proof}
These lattice are all even, unimodular and extremal and their dimension is 
divisible by $24$ hence their sets of minimal vectors form $11$-designs by a 
theorem by Venkov (see e.g. \cite[Chapter 7, Theorem 23]{c&s}).
\end{proof}

\section{$11$-design lattices with minimum $\le 9$}
Throughout this section $\Lambda\subseteq\R^n$ denotes an integral $11$-design
lattice with minimum $m\le 9$ and $s=|X|$ with $X\disj -X := S(\Lambda)$.
We will proceed in this section with the proof of theorem \ref{main}
part \ref{b} and compute the possible values for the dimension
and the kissing number in the same way as in Lemma \ref{9-des}.
\begin{lemma}
Using the definitions in the proof of Lemma \ref{9-des} 
we get that $s_i=0$ for $i>4$. The following system of linear 
equations has non-negative integral solutions for 
the $s_i$ and for $s:=|X|$ if $S(\Lambda)$ is a spherical $11$-design:
\begin{align*}
\begin{pmatrix} 1&2^2&3^2&4^2\\1&2^4&3^4&4^4\\1&2^6&3^6&4^6\\
	1&2^8&3^8&4^8\\1&2^{10}&3^{10}&4^{10}
\end{pmatrix}
\begin{pmatrix} s_1\\ s_2\\ s_3\\ s_4\end{pmatrix}=
\begin{pmatrix} \frac{sm^2}{n} -m^2\\ \frac{3sm^4}{n(n+2)}-m^4\\
	\frac{15sm^6}{n(n+2)(n+4)}-m^6\\ \frac{105sm^8}{n(n+2)(n+4)(n+6)}-m^8\\
	\frac{945sm^{10}}{n(n+2)(n+4)(n+6)(n+8)}-m^{10}
\end{pmatrix}.
\end{align*}
\end{lemma}
\begin{rem}
We get no solutions $(n,s,s_i)_{i\le4}\in \Z_{>0}^6$ for $m=9$ and for 
$m=8$ we get such solutions only for the values of $n$ and $s$ in 
Tabel \ref{11designs}.
\begin{table}[h]
\begin{center}
\begin{tabular}{|c|c|c|c|c|c|}
\hline
n & 50       & 56        & 62         & 64         & 66       \\ 
\hline
s & 57256875 & 237875400 & 1071285600 & 1866110400 & 3236535225\\
\hline
n & 68        &72 	   & 76         & 78        & 82\\
\hline
s & 474335190 & 3109087800 & 1263241980 & 866338200 & 470377215\\
\hline 
\end{tabular}
\end{center}
\caption{Dimensions and Kissing numbers for 11-design lattices.}
\label{11designs}
\end{table}
\end{rem}
Now we can see with the same arguments as in Lemma \ref{26_36} and 
Lemma \ref{not26_36} that an integral 11-design lattice 
has to be unimodular.
\begin{lemma}\label{56}
There is no non-unimodular lattice with minimum $8$ whose minimal 
vectors form a spherical $11$-design.
\end{lemma}
\begin{proof}
For all elements $v \in \Lambda^*\setminus\Lambda$ that are minimal 
in their class modulo $\Lambda$ we can define
$t_i(v):=|\{x\in S(\Lambda)|(x,v)=i\}|$. The $t_i$ are independent 
of $v$ and for $i>5$ $t_i=0$.
Therefore we get a system of equations again with $t:=(v,v)$:
\begin{align*}
\begin{pmatrix} 1&2^2&3^2&4^2\\1&2^4&3^4&4^4\\1&2^6&3^6&4^6\\
	1&2^8&3^8&4^8\\1&2^{10}&3^{10}&4^{10}
\end{pmatrix}
\begin{pmatrix} t_1\\ t_2\\ t_3 \\ t_4\\ \end{pmatrix}=
\begin{pmatrix} \frac{smt}{n}\\ \frac{3sm^2t^2}{n(n+2)}\\
	\frac{15sm^3t^3}{n(n+2)(n+4)}\\ \frac{105sm^4t^4}{n(n+2)(n+4)(n+6)}\\
	\frac{945sm^5t^5}{n(n+2)(n+4)(n+6)(n+8)}
\end{pmatrix}.
\end{align*}
$t$ has to be rational and positive. For every pair $(n,s)$ from 
Table \ref{11designs} we get a solution of the system and a polynomial equation of 
degree $5$ whose positive rational roots are the possible values for $t$.
The only dimension in which we get a positive rational value for $t$ is 
$n=56$ with $t=6$.
But then $\Lambda^*$ would have to be even and hence $\Lambda$ would 
be unimodular.
\end{proof}

\begin{lemma}
Let $\Lambda$ be unimodular with $\min(\Lambda)=8$ and $S(\Lambda)$
a spherical $11$-design, then $n=72$ and $\Lambda$ is even and extremal.
\end{lemma}
\begin{proof}
$\Lambda$ is even (see Lemma \ref{unimod}).
As the theta-series of even unimodular lattices are 
modular forms, $n$ has to be divisible by
eight and $\min(\Lambda)\le 2\lfloor\frac{n}{24}\rfloor+2$ \cite[V.2.8.Satz]{krieg}. 
Therefore $n=72$ for $m=8$ is the only possible combination. 
\end{proof}

\section{$13$-design lattices of minimum $\le 11$}
We will now prove that there is no integral lattice with minimum
smaller or equal to 11 whose minimal vectors form a 
$13$-design.
For minima smaller than $10$ we can use the results for $11$-designs.
\begin{lemma}
There is no integral lattice $\Lambda$ with $\min(\Lambda)<10$ such that
$S(\Lambda)$ is a spherical $13$-design.
\end{lemma}
\begin{proof}
If $S(\Lambda)$ forms a $13$-design it also forms an $11$-design and hence
can only be an extremal even unimodular lattice of dimension $24$, $48$ or 
$72$. But as a result of \cite[Proposition 4.1]{martinet} we know that these 
lattices yield no higher designs.
\end{proof}

So the only statement left to prove is the following:
\begin{lemma}
There is no integral $13$-design lattice of minimum $10$ or $11$.
\end{lemma}
\begin{proof}
If we assume that $\Lambda$ would be an integral $13$-design 
lattice with $\min(\Lambda)\in\{10,11\}$
then the following system of equations would have integral non-negative 
solutions for $s$ and $s_1,\dots ,s_5$.
\begin{align*}
\begin{pmatrix} 1&2^2&3^2&4^2&5^2\\1&2^4&3^4&4^4&5^4\\
	1&2^6&3^6&4^6&5^6\\1&2^8&3^8&4^8&5^8\\
	1&2^{10}&3^{10}&4^{10}&5^{10}\\
	1&2^{12}&3^{12}&4^{12}&5^{12}\\
\end{pmatrix}
\begin{pmatrix} s_1\\ s_2\\ s_3\\ s_4 \\ s_5\end{pmatrix}=
\begin{pmatrix} \frac{sm^2}{n} -m^2\\ \frac{3sm^4}{n(n+2)}-m^4\\
	\frac{15sm^6}{n(n+2)(n+4)}-m^6\\ \frac{105sm^8}{n(n+2)(n+4)(n+6)}-m^8\\
	\frac{945sm^{10}}{n(n+2)(n+4)(n+6)(n+8)}-m^{10}\\
	\frac{10395sm^{12}}{n(n+2)(n+4)(n+6)(n+8)(n+10)}-m^{12}
\end{pmatrix}.
\end{align*}
But an easy calculation shows that there are no such solutions.
\end{proof}

\end{document}